\documentclass[11pt]{amsart}
\usepackage{amscd,amsxtra,amssymb, bbm}
\usepackage[new]{old-arrows}

\makeatletter
\newcommand*{\rom}[1]{\expandafter\@slowromancap\romannumeral #1@}
\makeatother

\makeatletter
\newcommand{\doublewidetilde}[1]{{%
  \mathpalette\double@widetilde{#1}%
}}
\newcommand{\double@widetilde}[2]{%
  \sbox\z@{$\m@th#1\widetilde{#2}$}%
  \ht\z@=.9\ht\z@
  \widetilde{\box\z@}%
}
\makeatother

\usepackage{stackengine}
\stackMath
\newcommand\tsup[2][2]{%
 \def\useanchorwidth{T}%
  \ifnum#1>1%
    \stackon[-.5pt]{\tsup[\numexpr#1-1\relax]{#2}}{\scriptscriptstyle\sim}%
  \else%
    \stackon[.5pt]{#2}{\scriptscriptstyle\sim}%
  \fi%
}

\usepackage{multirow} 
\usepackage{stmaryrd} 

\input xy
\xyoption{all}
\xyoption{arc}
\SelectTips{cm}{10}

\newtheorem{theorem}{Theorem}[section]
\newtheorem{corollary}[theorem]{Corollary}
\newtheorem{lemma}[theorem]{Lemma}
\newtheorem{proposition}[theorem]{Proposition}

\theoremstyle{definition}

\newtheorem{remark}[theorem]{Remark}

\DeclareMathOperator{\Perf}{\mathsf{Perf}}

\newcommand{\kEnd}{\mbox{\emph{End}}}
\newcommand{\kHom}{\mbox{\emph{Hom}}}

\newcommand{\can}{\mathsf{can}}

\newcommand{\tr}{\mathsf{tr}}

\DeclareMathOperator{\res}{\mathsf{res}}
\DeclareMathOperator{\ev}{\mathsf{ev}}
\DeclareMathOperator{\Sol}{\mathsf{Sol}}

\DeclareMathOperator{\Coh}{\mathsf{Coh}}

\DeclareMathOperator{\Pic}{\mathsf{Pic}}

\DeclareMathOperator{\Hom}{\mathsf{Hom}}

\DeclareMathOperator{\Ext}{\mathsf{Ext}}
\DeclareMathOperator{\Lin}{\mathsf{Lin}}
\DeclareMathOperator{\GL}{\mathsf{GL}}
\DeclareMathOperator{\Aut}{\mathsf{Aut}}

\DeclareMathOperator{\End}{\mathsf{End}}
\DeclareMathOperator{\Mat}{\mathsf{Mat}}

\DeclareMathOperator{\lieg}{\mathfrak{g}}

\DeclareMathOperator{\lieA}{\mathfrak{A}}

\newcommand{\bul}{\scriptstyle{\bullet}}

\newcommand{\CC}{\mathbb{C}}
\newcommand{\NN}{\mathbb{N}}
\newcommand{\ZZ}{\mathbb{Z}}

\newcommand{\EE}{\mathbb{E}}
\newcommand{\XX}{\mathbb{X}}

\newcommand{\PP}{\mathbb{P}}

\newcommand{\kMat}{\mbox{\emph{Mat}}}

\newcommand{\kA}{\mathcal{A}}

\newcommand{\kB}{\mathcal{B}}

\newcommand{\kE}{\mathcal{E}}
\newcommand{\kF}{\mathcal{F}}
\newcommand{\kG}{\mathcal{G}}
\newcommand{\kH}{\mathcal{H}}

\newcommand{\kO}{\mathcal{O}}
\newcommand{\kL}{\mathcal{L}}

\newcommand{\kK}{\mathcal{K}}

\newcommand{\kS}{\mathcal{S}}

\newcommand{\lar}{\longrightarrow}

\begin{document}

\title[On elliptic solutions of AYBE]{On elliptic solutions of the associative Yang--Baxter equation}

\author{Igor Burban}
\address{
Universit\"at Paderborn\\
Institut f\"ur Mathematik \\
Warburger Stra\ss{}e 100 \\
33098 Paderborn \\
Germany
}
\email{burban@math.uni-paderborn.de}

\author{Andrea Peruzzi}
\address{
Universit\"at Paderborn\\
Institut f\"ur Mathematik \\
Warburger Stra\ss{}e 100 \\
33098 Paderborn \\
Germany
}
\email{aperuzzi@math.uni-paderborn.de}

\begin{abstract}
We give a direct proof of the fact  that elliptic solutions of the associative Yang--Baxter equation arise from an appropriate spherical order on an elliptic curve.
\end{abstract}

\maketitle

\section{Introduction}
Let $\lieA = \Mat_{n}(\CC)$ be the algebra of square matrices of size $n \in \NN$ and 
$\bigl(\CC^3, 0) \stackrel{r}\lar \lieA \otimes \lieA$ be the germ of a meromorphic function. The following version of the associative Yang--Baxter equation (AYBE) with spectral parameters was introduced by Polishchuk in \cite{Polishchuk1}:
\begin{multline}\label{E:AYBE}
r(u; x_1, x_2)^{12} r(u+v; x_2, x_3)^{23}  =
r(u+v;x_1, x_3)^{13} r(-v; x_1, x_2)^{12} + \\
r(v; x_2, x_3)^{23} r(u; x_1, x_3)^{13}.
\end{multline}
The upper indices in this equation   indicate the corresponding  embeddings of
$\lieA \otimes \lieA$ into $\lieA \otimes \lieA \otimes \lieA$. For example,
the germ $r^{13}$ is defined as
$$
r^{13}: \mathbb{C}^3 \stackrel{r}\lar \lieA \otimes \lieA
\stackrel{\imath_{13}}\lar \lieA \otimes \lieA \otimes \lieA,
$$
where $\imath_{13}(x\otimes y) = x \otimes 1 \otimes y$. Two other
germs $r^{12}$ and $r^{23}$  are defined in a similar way.
We are interested in those solutions of AYBE,  which are  non-degenerate,  skew-symmetric  (meaning that  
$r(v; x_1, x_2) = -r^{21}(-v; x_2, x_1)$) and which admit a Laurent expansion of the form
\begin{equation}\label{E:AnsatzAYBE}
r(v; x_1, x_2) =
\frac{\mathbbm{1} \otimes \mathbbm{1}}{v} + r_{0}(x_1, x_2) + v r_1(x_1, x_2) +
v^2 r_2(x_1, x_2) +\dots
\end{equation}
All elliptic and trigonometric solutions of AYBE satisfying (\ref{E:AnsatzAYBE}) were classified in \cite{Polishchuk1, Polishchuk2}. Recall the description  of elliptic solutions of AYBE. 

\smallskip
\noindent
Let $\varepsilon = \exp\bigl(\frac{2 \pi i d}{n}\bigr)$, where $0 < d < n$ is such that $\gcd(d, n) = 1$. We put
\begin{equation}\label{E:matricesXY}
X = \left(
\begin{array}{cccc}
1 & 0 & \dots & 0 \\
0 & \varepsilon & \dots & 0 \\
\vdots & \vdots & \ddots & \vdots \\
0 & 0 & \dots & \varepsilon^{n-1}
\end{array}
\right)
\quad
\mbox{and} \quad
Y = \left(
\begin{array}{cccc}
0 & 1 & \dots & 0 \\
\vdots & \vdots & \ddots & \vdots \\
0 & 0 & \dots & 1 \\
1 & 0 & \dots & 0
\end{array}
\right).
\end{equation}
For any $(k, l) \in I := \bigl\{1, \dots, n\bigr\} \times \bigl\{1, \dots, n\bigr\}$ denote
$
Z_{(k, l)} = Y^k X^{-l}$ and $
Z_{(k, l)}^\vee  = \frac{1}{n} X^{l} Y^{-k}.
$
Then the following expression
\begin{equation}\label{E:SolutionElliptic}
r_{((n, d), \tau)}(v; x_1, x_2)  = \sum\limits_{(k, l) \in I}
\exp\Bigl(\frac{2 \pi i d}{n} k x\Bigr)
\sigma\Bigl(v + \frac{d}{n}\bigl(k \tau + l\bigr),  x\Bigr)
Z_{(k, l)}^\vee \otimes Z_{(k, l)}
\end{equation}
is a solution of AYBE satisfying (\ref{E:AnsatzAYBE}), where  $x = x_2 - x_1$ and
\begin{equation}\label{E:KroneckerEllFunct}
\sigma(a,  z) =
 2 \pi i \sum\limits_{n \in \ZZ} \frac{\exp(- 2 \pi i n z)}{1 - \exp\bigl(-2 \pi i (a - 2 \pi i n \tau)\bigr)}
\end{equation} is the Kronecker  elliptic function \cite{Weil} for $\tau \in \CC$  such that $\mathsf{Im}(\tau) > 0$. See also \cite[Section III]{LOZ} for a direct proof of this fact.

 In his recent work \cite{Polishchuk3} Polishchuk  showed that non-degenerate skew-symmetric solutions of AYBE  satisfying (\ref{E:AnsatzAYBE}) can be obtained from appropriate triple Massey products in the perfect derived category of coherent sheaves $\Perf(\EE)$ on a non-commutative projective curve $\EE = (E, \kA)$, where $E$ is an irreducible projective curve over $\CC$ of arithmetic genus one    and $\kA$ is a symmetric spherical order on $E$. 
 A simplest example of such an order is given by $\kA = \kEnd_E(\kF)$, where $\kF$ is a simple vector bundle on $E$. Let $E = E_\tau := \CC/\langle 1, \tau\rangle$ be the elliptic curve determined by $\tau \in \CC$ and $\kF$ be a simple vector bundle of rank $n$ and degree $d$ on $E$. It follows from results of Atiyah \cite{Atiyah} that such $\kF$ exists and the sheaf of algebras $\kA = \kA_{(n, d)} := \kEnd_E(\kF)$ does not depend on the choice of $\kF$. We show that the solution of AYBE arising from  the non-commutative projective curve $\bigl(E_\tau, \kA_{(n, d)}\bigr)$ is given by the formula (\ref{E:SolutionElliptic}). In  \cite[Section 2]{Polishchuk1}, the corresponding computations were performed using the homological mirror symmetry and explicit formulae for triple Massey products in the Fukaya category of a torus. The expression for the resulted solution of AYBE (see  \cite[formula (2.3)]{Polishchuk1})  was different   from (\ref{E:SolutionElliptic}). Our computations are straightforward  and  based by  techniques developed in the articles \cite{BK4, BH}.

\medskip
\noindent
\emph{Acknowledgement}. This  work  was supported  by the DFG project Bu--1866/5--1 as well as by CRC/TRR 191 project ``Symplectic
Structures in Geometry, Algebra and Dynamics'' of German Research Council (DFG).  We are grateful to Raschid Abedin for discussions of results of this paper.

\section{Symmetric spherical orders on curves of genus one and AYBE}
In this section we make a brief review of Polishchuk's construction \cite{Polishchuk3}.
Let $E$ be an irreducible projective curve over $\CC$ of arithmetic genus one, $\breve{E}$ its smooth part, $\kO$ its structure sheaf, $\kK$ the sheaf of rational functions on $E$  and $\Omega$ the sheaf   of regular differential one-forms on $E$.  
There exists a regular differential one-form  $\omega \in \Gamma(E, \Omega)$ such that 
$\Gamma(E, \Omega) = \CC \omega$. Such $\omega$ also defines an isomorphism $\kO \cong \Omega$. If $E$ is singular then it  is rational. In this case, let 
$\PP^1 \stackrel{\nu}\to E$ be the normalization morphism  and $\widetilde\kO = 
\nu_\ast\bigl(\kO_{\PP^1}\bigr)$. 

Let $\kA$ be a sheaf of orders on $E$. By definition, $\kA$ is a torsion free coherent sheaf of $\kO$-algebras on $E$ such that  $\kA \otimes_{\kO} \kK \cong \kMat_n(\kK)$ for some $n \in \NN$.  For any order $\kA$ we have the canonical  trace morphism $\kA \stackrel{t}\to \widetilde\kO$, which coincides with the restriction of the trace morphism $\kA \hookrightarrow \kA \otimes_{\kO} \kK\cong \kMat_n(\kK) \stackrel{t}\to \kK$ (if $E$ is smooth then $\widetilde\kO = \kO$). Following \cite{Polishchuk3}, the order  $\kA$ is called  \emph{symmetric spherical} if the following conditions are fulfilled:
\begin{itemize}
\item The image of the trace morphism $t$ is  $\kO$ and the induced morphism of coherent sheaves $\kA \stackrel{t^\sharp}\lar \kA^\vee := \kHom_{E}(\kA, \kO)$ is an isomorphism.
\item We have: $\Gamma(E, \kA) \cong \CC$. 
\end{itemize}

Consider the non-commutative projective curve  $\EE = (E, \kA)$. Let   $\Coh(\EE)$ be the category of coherent sheaves on $\EE$ (these are sheaves of $\kA$-modules which are coherent as $\kO$-modules) and $\Perf(\EE)$ be the corresponding perfect derived category. Recall that $\Omega_{\EE} := \kHom_E(\kA, \Omega)$ is a dualising bimodule of $\EE$. If $\kA$ is symmetric then 
$\Omega_{\EE} \cong \kA$ as $\kA$-bimodules and $\Perf(\EE)$ is a triangulated 1-Calabi--Yau category. The last assertion  means that for any pair of objects $\kG^{\bul}, \kH^{\bul}$ in $\Perf(\EE)$ there is an isomorphism of vector spaces 
\begin{equation}\label{E:SerreDuality}
\Hom_{\EE}\bigl(\kG^{\bul}, \kH^{\bul}\bigr) \cong \Hom_{\EE}\bigl(\kH^{\bul}, \kG^{\bul}[1]\bigr)^\ast 
\end{equation}
which is functorial in both arguments.

Let $P = \underline{\Pic}^0(E)$ be the Jacobian of $E$ and $\kL \in \Pic(P \times E)$ be  a universal line bundle. For any $v \in P$, let $\kL^{v}:= \kL\big|_{\{v\} \times E} \in \Pic^{0}(E)$  and $\kA^{v} := \kA \otimes_\kO \kL^{v} \in \Coh(\EE)$.

\begin{lemma}
The coherent sheaf $\kA$ is semi-stable of slope zero. Moreover,
\begin{equation}\label{E:VanishingLocus}
\Gamma(E, \kA^{v}) = 0 = H^1(E, \kA^{v})
\end{equation} for all  but finitely many points $v \in P$. 
\end{lemma}
\begin{proof} Let $\kB$ be the kernel of the trace morphism $\kA \stackrel{t}\to \kO$. It follows from the long exact cohomology sequence of 
$
0 \to \kB \to\kA \stackrel{t}\to \kO \to 0
$
 that $H^0(E, \kB) = 0 = H^1(E, \kB)$. Hence, $\kB$ is a semi-stable coherent sheaf on 
 $E$ of slope zero and $\kA \cong \kB \oplus \kO$. It follows that $\kA$ is semi-stable, too. The latter fact also implies the vanishing $\Gamma(E, \kA^{v}) = 0 = H^1(E, \kA^{v})$ for all but finitely many  $v \in P$. 
\end{proof}

\begin{corollary}
There exists a proper closed subset $D \subset P \times P$ such that 
\begin{equation}\label{E:Vanish}
\Hom_{\EE}\bigl(\kA^{v_1}, \kA^{v_2}\bigr) = 0 = \Ext^1_{\EE}\bigl(\kA^{v_1}, \kA^{v_2}\bigr)
\end{equation}
for all $v_1, v_2 \in (P \times P) \setminus D$. 
\end{corollary}
\begin{proof} This statement follows from the isomorphisms
\begin{equation}
\Ext^i_{\EE}\bigl(\kA^{v_1}, \kA^{v_2}\bigr) \cong H^i\bigl(E, \kEnd_{\EE}(\kA^{v_1}, \kA^{v_2})\bigr) \cong H^i\bigl(E, \kA^{v_2 - v_1}\bigr), \quad i = 0, 1
\end{equation}
and the vanishing (\ref{E:VanishingLocus}).
\end{proof}

\noindent
Recall that for any $x, y \in \breve{E}$
we have the following standard  short exact sequences
\begin{equation}\label{E:CanonicalSequences}
0 \lar \Omega \lar \Omega(x) \xrightarrow{\underline{\res}_x} \CC_x \lar 0 \quad \mbox{\rm and} \quad 0 \lar \kO(-y) \lar \kO \xrightarrow{\underline{\ev}_y} \CC_y \lar 0,
\end{equation}
where $\underline{\res}_x$ and $\underline{\ev}_y$ are the residue and evaluation morphisms, respectively. Using the isomorphism 
$\kO \stackrel{\omega}\to \Omega$,  we can rewrite the first short exact sequence as  
\begin{equation}\label{E:ResidueTrivialized}
0 \lar \kO \lar \kO(x) \xrightarrow{\underline{\res}^\omega_x} \CC_x \lar 0.
\end{equation}
For any $\kH \in \Coh(\EE)$ denote $\kH\big|_{x} := \kH \otimes_\kO \CC_x \in \Coh(\EE)$.  Tensoring (\ref{E:ResidueTrivialized}) by $\kA^v$ (where 
 $v \in P$ is an arbitrary point), we get the following short exact sequence in $\Coh(\EE)$: 
$$
0 \lar \kA^{v} \lar \kA^{v}(x) \lar \kA^{v}\big|_x \lar 0.  
$$
Next, for any $(u, v) \in P \times P$ we have the induced long  exact sequence of vector spaces
$$
0 \lar \Hom_{\EE}(\kA^{u}, \kA^{v}\bigr) \lar 
\Hom_{\EE}(\kA^{u}, \kA^{v}(x)\bigr) \lar \Hom_{\EE}(\kA^{u}, \kA^{v}\big|_x\bigr) \lar 
\Ext^1_{\EE}(\kA^{u}, \kA^{v}\bigr).
$$
It follows from (\ref{E:VanishingLocus})  that the  linear map
\begin{equation}\label{E:ResidueSheafA}
\Hom_{\EE}(\kA^{u}, \kA^{v}(x)\bigr) \xrightarrow{\underline{\res}^{\kA}(u,v; x)} \Hom_{\EE}(\kA^{v}\big|_x, \kA^{v}\big|_x\bigr)
\end{equation}
is an isomorphism if $(u, v) \in (P \times P)\setminus D$.

\smallskip
\noindent
Similarly, for any $v \in P$ and $x \ne y \in \breve{E}$ we have the following  short exact sequence 
$$
0 \lar \kA^v(x-y) \lar \kA^v(x) \lar \kA^v(x)\big|_y \lar 0
$$
in $\Coh(\EE)$, which is obtained by tensoring the evaluation sequence (\ref{E:CanonicalSequences}) by $\kA^v(x)$. After applying to it the functor 
$\Hom_{\EE}(\kA^u, \,-\,)$ and using  a canonical isomorphism $\kA^v\big|_y \cong \kA^v(x)\big|_y$, we obtain a linear map
\begin{equation}\label{E:EvaluationSheafA}
\Hom_{\EE}(\kA^{u}, \kA^{v}(x)\bigr) \xrightarrow{\underline{\ev}^{\kA}(u,v; x, y)} \Hom_{\EE}(\kA^{u}\big|_y, \kA^{v}\big|_y\bigr).
\end{equation}
Let $\Hom_{\EE}\bigl(\kA^u\big|_x,  \kA^v\big|_x\bigr) \xrightarrow{\alpha(u, v; x,y)}
\Hom_{\EE}\bigl(\kA^u\big|_y,  \kA^v\big|_y\bigr)$ be (the unique) linear map making the following diagram of vector spaces 
\begin{equation}\label{E:ResidueEvaluation}
\begin{array}{c}
\xymatrix@C-C@R+3mm
{
& \Hom_{\EE}\bigl(\kA^u, \kA^v(x)\bigr)
\ar[ld]_-{\underline{\res}^{\kA}(u,v; x)\phantom{x}}
\ar[rd]^-{\phantom{x}\underline{\ev}^{\kA}(u,v; x, y)} &
\\
\Hom_{\EE}\bigl(\kA^u\big|_x,  \kA^v\big|_x\bigr)
\ar[rr]^-{\alpha(u, v; x,y)} & &
\Hom_{\EE}\bigl(\kA^u\big|_y,  \kA^v\big|_y\bigr)
}
\end{array}
\end{equation}
commutative. Let $\gamma(u, v; x, y)\in \Hom_{\EE}\bigl(\kA^v\big|_x,  \kA^u\big|_x\bigr) \otimes \Hom_{\EE}\bigl(\kA^u\big|_y,  \kA^v\big|_y\bigr)$ be the image of $\alpha(u, v; x,y)$ under the composition of the following canonical isomorphisms of vector spaces:
\begin{multline*}
\Lin\Bigl(\Hom_{\EE}\bigl(\kA^u\big|_x,  \kA^v\big|_x\bigr), 
\Hom_{\EE}\bigl(\kA^u\big|_y,  \kA^v\big|_y\bigr)\Bigr) \cong \Hom_{\EE}\bigl(\kA^u\big|_x,  \kA^v\big|_x\bigr)^\ast \otimes \Hom_{\EE}\bigl(\kA^u\big|_y,  \kA^v\big|_y\bigr) \\
\cong \Hom_{\EE}\bigl(\kA^v\big|_x,  \kA^u\big|_x\bigr) \otimes \Hom_{\EE}\bigl(\kA^u\big|_y,  \kA^v\big|_y\bigr),
\end{multline*}
where the last isomorphism is induced by the trace morphism $t$.

\smallskip
\noindent
 Let 
$P \times P \stackrel{\eta}\lar  P$ be the group operation on $P$ and  $o \in P$ be the corresponding neutral element (i.e. $\kO  \cong  \kL^o$).  Consider the canonical projections 
 $P \times P \times E \stackrel{\pi_i}\lar P \times E, \, (x_1, x_2; x) \mapsto (x_i, x)$ for $i = 1, 2$ and   $P \times P \times E \stackrel{\pi_\circ}\lar P \times E, \, (x_1, x_2; x) \mapsto (x_1, x_2)$. Then there exists $\kS \in \Pic(P \times P)$ such that 
 \begin{equation}\label{E:UnivProp}
 (\eta \times \mathbbm{1})^\ast \kL \cong \pi_1^\ast \kL \otimes 
 \pi_2^\ast \kL \otimes \pi_\circ^\ast \kS. 
 \end{equation} 
In particular, $\kL^{v_1} \otimes \kL^{v_2} \cong \kL^{v_1+v_2}$, where $v_1 + v_2 = \eta(v_1, v_2)$.

\smallskip
\noindent
For any type of $E$ (elliptic, nodal or cuspidal) there exists  a complex analytic covering map
$(\CC, +) \stackrel{\chi}\lar (P, \eta)$, which is also a group homomorphism. In this way we get a local coordinate on $P$ in a neighbourhood of $o$. Next, we put:
$\overline\kL := (\chi \times \mathbbm{1})^\ast \kL$. Since any line bundle on $\CC \times \CC$ is trivial, we get from  (\ref{E:UnivProp}) an induced isomorphism 
\begin{equation}\label{E:UnivProp2}
 (\bar\eta \times \mathbbm{1})^\ast \overline\kL \cong \bar\pi_1^\ast \overline\kL \otimes 
 \bar\pi_2^\ast \overline\kL,
 \end{equation} 
 where $\bar\eta$ (respectively, $\bar\pi_i$) is the composition of $\eta$ (respectively, $\pi_i$) with $\chi \times \chi$. It follows that we have isomorphisms
\begin{equation}\label{E:SectionsAlphaBeta}
 \kO_E \stackrel{\alpha}\lar \overline\kL\Big|_{0 \times E} \quad \mbox{and} \quad
 \kO_{\CC \times \CC \times E} \stackrel{\beta}\lar \bar\eta^\ast \overline\kL^\vee \otimes \bar\pi_1^\ast \overline\kL \otimes 
 \bar\pi_2^\ast \overline\kL.
\end{equation}

\smallskip
\noindent
Let $U \subset \breve{E}$ be an open subset for which there exists an isomorphism  of $\Gamma(U, \kO_E)$-algebras
\begin{equation}\label{E:TrivialisationXi}
\Gamma(U, \kA) \stackrel{\xi}\lar \lieA \otimes_{\CC} \Gamma(U, \kO_E)
\end{equation}
as well as a trivialization 
\begin{equation}\label{E:TrivialisationZeta}
\Gamma(\CC \times U, \overline\kL) \stackrel{\bar\zeta}\lar \Gamma(\CC \times U, \kO_{\CC \times E}),
\end{equation} 
which identify the sections $\alpha$ and $\beta$ from (\ref{E:SectionsAlphaBeta}) with the identity section. Since $\eta$ is a complex analytic covering map, we get from $\bar\zeta$ a local trivialization $\zeta$ of the universal family $\kL$.
Then such trivializations $\xi$ and $\zeta$ allow to identify $\gamma(u, v; x, y)$  with a tensor $\rho(u, v; x, y) \in \lieA \otimes \lieA$. Note that by the construction the tensor $\rho(u, v; x, y)$ depends only  the difference $w := u-v \in P$ with respect to the group law on the Jacobian $P$. 

\begin{theorem}[Polishchuk \cite{Polishchuk3}]
The constructed tensor $\varrho(w; x, y) = \rho(u, v; x, y)$ is a non-degenerate skew-symmetric solution of the associatiove Yang--Baxter equation (\ref{E:AYBE}). 
\end{theorem}

\smallskip
\noindent
Recall the key steps  of the proof of this result. For any $x \in \breve{E}$,  let $\kS^x \in \Coh(\EE)$ be a simple  object of finite length supported at $x$ (which is unique, up to an isomorphism).  For any $(u, v) \in 
(P \times P) \setminus D$ and $(x, y) \in (\breve{E} \times \breve{E}), x \ne y$ consider  the triple Massey product
$$
\Hom_{\EE}\bigl(\kA^u, \kS^x) \otimes \Ext^1_{\EE}(\kS^x, \kA^v) \otimes \Hom_{\EE}\bigl(\kA^v, \kS^y) \xrightarrow{m_3(u, v; x, y)} \Hom_{\EE}\bigl(\kA^u, \kS^y)
$$
in the triangulated category $\Perf(\EE)$. Since
$\Ext^1_{\EE}(\kS^x, \kA^v)^\ast \cong \Hom_{\EE}\bigl(\kA^v, \kS^x)$ (see (\ref{E:SerreDuality})), we get from $m_3(u, v; x, y)$ a linear map
\begin{equation}
 \Hom_{\EE}\bigl(\kA^u, \kS^x) \otimes \Hom_{\EE}\bigl(\kA^v, \kS^y)
\xrightarrow{m^{u, v}_{x, y}} \Hom_{\EE}\bigl(\kA^v, \kS^x) \otimes \Hom_{\EE}\bigl(\kA^u, \kS^y).
\end{equation}
The constructed family of maps $m^{u, v}_{x, y}$ satisfies the identity
\begin{equation}\label{E:AYBEcat}
(m^{v_3, v_2}_{x_1, x_2})^{12} (m^{v_1, v_3}_{x_1, x_3})^{13} -
(m^{v_1, v_3}_{x_2, x_3})^{23} (m^{v_1, v_2}_{x_1, x_2})^{12} +
(m^{v_1, v_2}_{x_1, x_3})^{13} (m^{v_2, v_3}_{y_2, y_3})^{23} = 0,
\end{equation}
both sides of which are viewed as linear maps
$$
\Hom_{\EE}(\kA^{v_1}, \kS^{x_1}) \otimes \Hom_{\EE}(\kA^{v_2}, \kS^{x_2})  \otimes
\Hom_{\EE}(\kA^{v_3}, \kS^{x_3}) \lar
$$
$$
\lar \Hom_{\EE}(\kA^{v_2}, \kS^{x_1}) \otimes \Hom_{\EE}(\kA^{v_3}, \kS^{x_2})  \otimes
\Hom_{\EE}(\kA^{v_1}, \kS^{x_3}).
$$
Moreover,  $m^{u, v}_{x, y}$  is
non-degenerate and skew-symmetric:
\begin{equation}\label{E:SkewSymmcat}
\iota({m}^{u, v}_{x, y}) = - {m}^{v, u}_{y, x},
\end{equation}
where $\iota$ is the isomorphism $$
 \Hom_{\EE}\bigl(\kA^u, \kS^x) \otimes \Hom_{\EE}\bigl(\kA^v, \kS^y) \lar \Hom_{\EE}\bigl(\kA^v, \kS^y) \otimes \Hom_{\EE}\bigl(\kA^u, \kS^x)
$$
 given by
$\iota(f\otimes g) = g \otimes f$. Both identities (\ref{E:AYBEcat}) and (\ref{E:SkewSymmcat}) are consequences of existence of an $A_\infty$-structure on $\Perf(\EE)$ which is cyclic with respect to the Serre duality (\ref{E:SerreDuality}).  Applying appropriate canonical isomorphisms, one can identify $m^{u, v}_{x, y}$ with the linear map $\alpha(u, v; x, y)$ from the commutative diagram (\ref{E:ResidueEvaluation}). See also \cite[Theorem 2.2.17]{BK4} for a detailed exposition in a similar setting.  \qed

\section{Solutions of AYBE as a section of a vector bundle}

\noindent
Following the work \cite{BK4}, we provide a global version of the commutative diagram (\ref{E:ResidueEvaluation}). Let $$B := P \times P \times \breve{E} \times \breve{E} \setminus 
\bigl(D \times \breve{E} \times \breve{E}\bigr) \cup \bigl(P \times P \times \Xi\bigr),$$ where $D \subset P \times P$ is the locus defined by (\ref{E:Vanish}) and $\Xi \subset \breve{E} \times \breve{E}$ is the diagonal. Let $X := B \times E$. Then the canonical projection 
$X \stackrel{\pi}\lar B$ admits two canonical sections $B \stackrel{\sigma_i}\lar X$ given by 
$\sigma_i(v_1, v_2; x_1, x_2) := (v_1, v_2; x_1, x_2; x_i)$ for $i = 1, 2$. Let 
$\Sigma_i := \sigma_i(B) \subset X$ be the corresponding Cartier divisor. Note that $\Sigma_1 \cap \Sigma_2 = \emptyset$. 

\smallskip
\noindent
Similarly to (\ref{E:ResidueTrivialized}), we have the following short exact sequence
in the category $\Coh(X)$:
\begin{equation}\label{E:ResidueGlobal}
0 \lar \kO_X \lar \kO_X(\Sigma_1) \xrightarrow{\underline{\res}^\omega_{\Sigma_1}} \kO_{\Sigma_1} \lar 0. 
\end{equation}
Here, for a local section $g(v_1, v_2; x_1, x_2;  x) = \dfrac{f(v_1, v_2; x_1, x_2;  x)}{x-x_1}$ of the line bundle $\kO_X(\Sigma_1)$ we put: 
$\underline{\res}^\omega_{\Sigma_1}(g)  = \res_{x = x_1}\bigl(g\omega_x)$, where $\omega_x$ is the pull-back of $\omega$ under the canonical projection
$X \stackrel{\pi_5}\lar E$.

Consider the non-commutative scheme $\XX = \bigl(X, \pi_5^\ast(\kA)\bigr)$ as well as coherent sheaves 
 $\kA^{(i)}:= \pi_5^\ast(\kA) \otimes \pi_{i, 5}^\ast(\kL) \in \Coh(\XX)$, where 
$X \stackrel{\pi_{i, 5}}\lar P \times E$ is  the canonical projection for
$i = 1, 2$. Tensoring (\ref{E:ResidueGlobal}) by $\kA^{(2)}$, we get a short exact sequence
\begin{equation}\label{E:ResidueInduced}
0 \lar \kA^{(2)} \lar \kA^{(2)}(\Sigma_1) \lar \kA^{(2)}\Big|_{\Sigma_1} \lar 0 
\end{equation}
in the category $\Coh(\XX)$. Since $\kA^{(1)}$ is a locally projective $\kO_{\XX}$-module, applying  the functor $\kHom_{\XX}(\kA^{(1)}, -)$ to (\ref{E:ResidueInduced}), we get an induced short exact sequence
\begin{equation}\label{E:ResidueInduced2}
0 \to \kHom_{\XX}\bigl(\kA^{(1)}, \kA^{(2)}\bigr) \lar \kHom_{\XX}\bigl(\kA^{(1)}, \kA^{(2)}(\Sigma_1)\bigr) \lar \kHom_{\XX}\Bigl(\kA^{(1)}, \kA^{(2)}\Big|_{\Sigma_1}\Bigr) \to 0
\end{equation}
in the category $\Coh(X)$. Base-change isomorphism combined with the vanishing  (\ref{E:Vanish}) imply that $R\pi_\ast\bigl(\kHom_{\XX}\bigl(\kA^{(1)}, \kA^{(2)}\bigr)\bigr) = 0$, where $R\pi_\ast: D^b\bigl(\Coh(X)\bigr) \lar D^b\bigl(\Coh(X)\bigr)$ is the derived direct 
image functor. Applying the functor $\pi_\ast$ to the short exact sequence (\ref{E:ResidueInduced2}),  we get the following isomorphism
$$
\pi_\ast\Bigl(\kHom_{\XX}\bigl(\kA^{(1)}, \kA^{(2)}(\Sigma_1)\bigr)\Bigr) \stackrel{\cong} 
\lar \pi_\ast \kHom_{\XX}\Bigl(\kA^{(1)}, \kA^{(2)}\Big|_{\Sigma_1}\Bigr)
$$
of coherent sheaves on $B$.
Since $\kHom_{\XX}\Bigl(\kA^{(1)}, \kA^{(2)}\Big|_{\Sigma_1}\Bigr) \cong 
\kHom_{\XX}\Bigl(\kA^{(1)}\Big|_{\Sigma_1}, \kA^{(2)}\Big|_{\Sigma_1}\Bigr)$,  we get an isomorphism $\underline{\res}^{\kA}_{\Sigma_1}$ of coherent sheaves on $B$ (which are even locally free) given as the composition
$$
\pi_\ast\Bigl(\kHom_{\XX}\bigl(\kA^{(1)}, \kA^{(2)}(\Sigma_1)\Bigr) \stackrel{\cong} 
\lar \pi_\ast \kHom_{\XX}\Bigl(\kA^{(1)}, \kA^{(2)}\Big|_{\Sigma_1}\Bigr) \stackrel{\cong}\lar
\pi_\ast \kHom_{\XX}\Bigl(\kA^{(1)}\Big|_{\Sigma_1}, \kA^{(2)}\Big|_{\Sigma_1}\Bigr).
$$
Next, we have the following short exact sequence of coherent sheaves on $X$: 
\begin{equation}\label{E:Resid}
0 \lar \kO_X(-\Sigma_2) \lar \kO_X \lar \kO_{\Sigma_2} \lar 0.
\end{equation}
Since $\Sigma_1 \cap \Sigma_2 = \emptyset$, the canonical morphism
$\kO_{\Sigma_2} \to \kO(\Sigma_1)\big|_{\Sigma_2}$ is an isomorphism. Tensoring (\ref{E:Resid}) by $\kA^{(2)}(\Sigma_1)$, we get a short exact sequence
\begin{equation}\label{E:EvalInduced}
0 \lar \kA^{(2)}(\Sigma_1 - \Sigma_2) \lar \kA^{(2)}(\Sigma_1) \lar \kA^{(2)}\Big|_{\Sigma_2} \lar 0 
\end{equation}
in the category $\Coh(\XX)$. Applying to (\ref{E:EvalInduced}) the functor $\kHom_{\XX}(\kA^{(1)}, -)$, we get an induced short exact sequence
$$
0 \to \kHom_{\XX}\bigl(\kA^{(1)}, \kA^{(2)}(\Sigma_1 - \Sigma_2)\bigr) \lar \kHom_{\XX}\bigl(\kA^{(1)}, \kA^{(2)}(\Sigma_1)\bigr) \lar \kHom_{\XX}\Bigl(\kA^{(1)}, \kA^{(2)}\Big|_{\Sigma_1}\Bigr) \to 0
$$
in the category $\Coh(X)$. Applying the functor $\pi_\ast$, we get a  morphism of locally free sheaves $\underline{\ev}^{\kA}_{\Sigma_2}$ on $B$ given as the composition
$$
\pi_\ast\Bigl(\kHom_{\XX}\bigl(\kA^{(1)}, \kA^{(2)}(\Sigma_1)\Bigr)  
\lar \pi_\ast \kHom_{\XX}\Bigl(\kA^{(1)}, \kA^{(2)}\Big|_{\Sigma_2}\Bigr) \stackrel{\cong}\lar
\pi_\ast \kHom_{\XX}\Bigl(\kA^{(1)}\Big|_{\Sigma_2}, \kA^{(2)}\Big|_{\Sigma_2}\Bigr).
$$
In other words, we get the following global version 
\begin{equation*}
\begin{array}{c}
\xymatrix@C-C@R+3mm
{
& \pi_\ast\Bigl(\kHom_{\XX}\bigl(\kA^{(1)}, \kA^{(2)}(\Sigma_1)\Bigr)
\ar[ld]_-{\underline{\res}^{\kA}_{\Sigma_1}\phantom{x}}
\ar[rd]^-{\phantom{x}\underline{\ev}^{\kA}_{\Sigma_2}} &
\\
\pi_\ast \kHom_{\XX}\Bigl(\kA^{(1)}\Big|_{\Sigma_1}, \kA^{(2)}\Big|_{\Sigma_1}\Bigr)
\ar[rr]^-{\alpha^\kA} & &
\pi_\ast \kHom_{\XX}\Bigl(\kA^{(1)}\Big|_{\Sigma_2}, \kA^{(2)}\Big|_{\Sigma_2}\Bigr)
}
\end{array}
\end{equation*}
of the the commutative diagram  (\ref{E:ResidueEvaluation}), where 
$\alpha^\kA := \underline{\ev}^{\kA}_{\Sigma_2} \circ \bigl(\underline{\res}^{\kA}_{\Sigma_1}\bigr)^{-1}$.

\smallskip
\noindent
For any $1 \le i, j \le 2$, consider the canonical projection
$$
P \times P \times \breve{E} \times \breve{E} \stackrel{\kappa_{ij}}\lar P \times E, \; (v_1, v_2; x_1, x_2) \mapsto (v_j, x_i).
$$
Then we have the following canonical isomorphism of coherent sheaves on $B$:
$$
\pi_\ast \kHom_{\XX}\Bigl(\kA^{(1)}\Big|_{\Sigma_i}, \kA^{(2)}\Big|_{\Sigma_i}\Bigr) \cong \kA^{\langle i\rangle} \otimes \kHom_B\bigl(\kappa^{\ast}_{1i}(\kL), \kappa^{\ast}_{2i}(\kL)\bigr),
$$
where $\kA^{\langle i\rangle}$ is the pull-back of $\kA$ on $B$ via the projection morphism $$
P \times P \times \breve{E} \times \breve{E} \lar E, \; (v_1, v_2; x_1, x_2) \mapsto x_i
$$
for $i = 1, 2$. The morphism of locally free $\kO_B$-modules
$$
\kA^{\langle 1\rangle} \otimes \kHom_B\bigl(\kappa^{\ast}_{11}(\kL), \kappa^{\ast}_{21}(\kL)\bigr)
\stackrel{\alpha^\kA}\lar \kA^{\langle 2\rangle} \otimes \kHom_B\bigl(\kappa^{\ast}_{12}(\kL), \kappa^{\ast}_{22}(\kL)\bigr)
$$
determines  a distinguished section 
\begin{equation}\label{E:AYBEasSection}
\gamma^\kA \in 
\Gamma\Bigl(B, 
\kA^{\langle 1\rangle} \otimes \kA^{\langle 2\rangle} \otimes \kappa_{11}^\ast(\kL) \otimes \kappa_{21}^\ast(\kL^\vee) \otimes \kappa_{22}^\ast(\kL) \otimes\kappa_{12}^\ast(\kL^\vee)\Bigr).
\end{equation}
For $i = 1, 2$  consider the canonical projections $P \times P \times E \stackrel{\psi_i}\lar P \times E, (v_1, v_2; x) \mapsto (v_i; x)$ as well as $P \times P \times E \stackrel{\psi}\lar P \times P, (v_1, v_2; x) \mapsto (v_1, v_2)$. Then there exists $\kS \in \Pic(P \times P)$ such that 
$$
\psi_1^\ast(\kL)  \otimes \psi_2^\ast(\kL^\vee) \cong (\mu \times \mathbbm{1})^\ast \otimes \psi^\ast(\kS),
$$
where  $P \times P \stackrel{\mu}\lar P, (v_1, v_2) \mapsto v_1-v_2$.
Finally, for  $i  = 1, 2$ consider the morphism 
$$
P \times P \times \breve{E} \times \breve{E} \stackrel{\mu_i}\lar P \times \breve{E}, (v_1, v_2; x_1, x_2) \mapsto (v_1 - v_2; x_i).
$$
Then we have an isomorphism of locally free sheaves
$$
\kappa_{11}^\ast(\kL) \otimes \kappa_{21}^\ast(\kL^\vee) \otimes \kappa_{22}^\ast(\kL) \otimes\kappa_{12}^\ast(\kL^\vee) \cong \mu_1^\ast(\kL) \otimes \mu_2^\ast(\kL^\vee).
$$
In these terms, we can regard  $\gamma^\kA$ from (\ref{E:AYBEasSection}) as a section
\begin{equation}\label{E:AYBEasSection2}
\gamma^\kA \in 
\Gamma\Bigl(B, 
\kA^{\langle 1\rangle} \otimes \kA^{\langle 2\rangle} \otimes \mu_1^\ast(\kL) \otimes \mu_2^\ast(\kL^\vee)\Bigr).
\end{equation}
Applying trivialisations $\xi$ of $\kA$ (see (\ref{E:TrivialisationXi})) and $\zeta$ of $\kL$ (see (\ref{E:TrivialisationZeta})), we obtain from $\gamma^\kA$ a tensor-valued function
$$
\rho_{\xi, \zeta}^\kA:   V \times V \times U \times U \lar \lieA \otimes \lieA,
$$
which satisfies the translation  property 
$$\rho_{\xi, \zeta}^\kA(v_1+u, v_2 + u; x_1, x_2) = \varrho_{\xi, \zeta}^\kA(v_1, v_2; x_1, x_2).$$ Recall that for all types of the genus one curve $E$ (smooth, nodal or cuspidal) we have a group homomorphism $(\CC, +) \lar (P, +)$, which is locally a biholomorphic map. 

\smallskip
\noindent
After making these identifications, we get the germ of a meromorphic function  
\begin{equation}(\CC^3, 0) \stackrel{\varrho}\lar \lieA \otimes \lieA, \quad \mbox{where} \quad \varrho(v_1 - v_2; x_1, x_2) := \rho^\kA_{\xi, \zeta}(v_1, v_2; x_1, x_2).
\end{equation}
This function is a non-degenerate skew-symmetric solution of AYBE.

\smallskip
\noindent
\textbf{Summary}. Let $\EE = (E, \kA)$ be a non-commutative projective curve, where   $E$ is an irreducible projective curve of arithmetic genus one and $\kA$ be a symmetric spherical order on $E$.
Let 
 $P$ be the  Jacobian of $E$ and  $\kL$  be a universal family of degree zero line bundles on $E$. Then we have a distinguished section $\gamma^\kA \in 
\Gamma\Bigl(B, 
\kA^{\langle 1\rangle} \otimes \kA^{\langle 2\rangle} \otimes \mu_1^\ast(\kL) \otimes \mu_2^\ast(\kL^\vee)\Bigr)$. Choosing trivializations $\xi$ of $\kA$ (see (\ref{E:TrivialisationXi})) and $\zeta$ of $\kL$ (see (\ref{E:TrivialisationZeta})), we get the germ of a meromorphic function $(\CC^3, 0) \stackrel{\varrho}\lar \lieA \otimes \lieA$, which is a non-degenerate skew-symmetric solution of AYBE.
A different trivialization $\tilde\xi$ of $\kA$ leads to a gauge-equivalent solution
$(\varphi(x_1) \otimes \varphi(x_1)\bigr) \varrho(v; x_1, x_2)$, where  
$(\CC, 0) \stackrel{\varphi}\lar \Aut_{\CC}(\lieA)$ is the germ of $\tilde{\xi} \xi^{-1}$.  Analogously, another choice of a trivialization $\zeta$ leads to an equivalent solution
$\exp\bigl(v (\beta(x_1)-\beta(x_2))\bigr)\varrho(v; x_1, x_2)$ for some 
holomorphic $(\CC, 0) \stackrel{\beta}\lar \CC$.

\begin{remark}
The simplest example of a symmetric spherical order  is $\kA = \kEnd_E(\kF)$, where $\kF$ is a simple vector bundle on $E$ of rank $n$ and degree $d$. It follows from  \cite{Atiyah, Burban1, BodnarchukDrozd} that such $\kF$ exists if and only if $n$ and $d$ are coprime and the sheaf of algebras $\kA = \kA_{(n, d)}  = \kEnd_E(\kF)$ does not depend on the choice of $\kF$. Moreover, according to \cite[Proposition 1.8.1]{Polishchuk3}, any symmetric spherical order on an elliptic curve $E$ is isomorphic to $\kA_{(n, d)}$ for some $0 < d < n$  mutually prime. 
\end{remark}

\begin{remark} Let $(u, v) \in P \times P \setminus D$ and 
$(x, y) \in \breve{E} \times \breve{E} \setminus \Xi$. Then we have  canonical isomorphisms 
$$\Hom_{\EE}\bigl(\kA^u\big|_x,  \kA^v\big|_x\bigr) \cong 
H^0(E, \kHom_{\EE}\bigl(\kA^u\big|_x,  \kA^v\big|_x\bigr) \cong H^0\bigl(E, \kA^{v-u}([x])\bigr).
$$
Analogously, we have canonical isomorphisms
$$
\Hom_{\EE}\bigl(\kA^u\big|_x,  \kA^v\big|_x\bigr) \cong \kA^{v-u}\big|_x
\quad \mbox{\rm and} \quad \Hom_{\EE}\bigl(\kA^u\big|_y,  \kA^v\big|_y\bigr) \cong \kA^{v-u}\big|_y
$$
such that the following diagram 
\begin{equation}\label{E:ResidueEvaluationRewritten}
\begin{array}{c}
\xymatrix@C-C@R+3mm
{
\Hom_{\EE}\bigl(\kA^u\big|_x,  \kA^v\big|_x\bigr) \ar[d]_-\cong & & & & & & & & \Hom_{\EE}\bigl(\kA^u, \kA^v(x)\bigr) \ar[d]_-\cong
\ar[llllllll]_-{\underline{\res}^{\kA}(u,v; x)\phantom{x}}
\ar[rrrrrrrr]^-{\phantom{x}\underline{\ev}^{\kA}(u,v; x, y)} & & & & & & & & \Hom_{\EE}\bigl(\kA^u\big|_y,  \kA^v\big|_y\bigr) \ar[d]^-\cong
\\
\kA^{v-u}\big|_x  & & & & & & & & \ar[llllllll]_-{\res^{v-u}_x} H^0\bigl(E, \kA^{v-u}([x])\bigr) \ar[rrrrrrrr]^-{\ev^{v-u}_y} & & & & & & & &  \kA^{v-u}\big|_y
}
\end{array}
\end{equation}
is commutative. Here, the linear maps $\res^{v-u}_x$ and $\ev^{v-u}_y$
are induced by the standard short exact sequences (\ref{E:CanonicalSequences}).
\end{remark}

\section{Elliptic solutions of AYBE}

\smallskip
\noindent
Let $\tau \in \CC$ be such that $\mathsf{Im}(\tau) > 0$, $\CC \supset \Lambda  = \langle 1, \, \tau\rangle \cong \ZZ^2$ and $E = E_\tau = \CC/\Lambda$. Recall some  standard  techniques to deal with holomorphic vector bundles on complex tori.  An
\emph{automorphy factor} is a pair
$(A, V)$, where $V$ is a finite dimensional vector space over $\CC$ and
$
A: \Lambda \times \CC \lar \GL(V)
$
is a holomorphic function such that $A(\lambda + \mu, z) = A(\lambda, z + \mu) A(\mu, z)$ for all
$\lambda, \mu \in \Lambda$ and $z \in \CC$. Such  $(A, V)$ defines the following  holomorphic vector bundle on
the torus $E$:
\begin{equation*}
\kE(A, V) := \CC \times V/\sim, \;
\mbox{where} \quad (z, v) \sim \bigl(z + \lambda, A(\lambda, z) v\bigr) \quad
\mbox{for all} \quad (\lambda, z, v) \in \Lambda \times \CC \times V.
\end{equation*}
Given two automorphy factors $(A, V)$ and $(B, V)$, the corresponding  vector bundles $\kE(A, V)$ and $\kE(B, V)$ are isomorphic if and only if
there exists a holomorphic function $H: \CC \rightarrow \GL(V)$ such that
$$
B(\lambda, z) = H(z + \lambda) A(\lambda, z) H(z)^{-1} \quad
\mbox{for all} \quad (\lambda, z) \in \Lambda\times \CC.
$$

 Let $\Phi: \CC \longrightarrow
\GL_n(\CC)$ be a holomorphic function such that $\Phi(z+1) = \Phi(z)$ for all $z \in \CC$.
Then one can define the automorphy factor
$(A, \CC^n)$ in the following way.

\medskip
\noindent
$-$
$A(0, z) = I_n$ (the identity $n \times n$ matrix).

\medskip
\noindent
$-$ For any $k \in \NN_{0}$ we set:
$$
A(k\tau, z) = \Phi\bigl(z + (k-1)\tau\bigr) \dots  \Phi(z) \;
\mbox{and} \; A(-k\tau, z) = A(k \tau, z-k \tau)^{-1}.
$$

\smallskip
\noindent
For  a proof of the following result, see \cite[Proposition 5.1]{BH}.

\begin{proposition} Let $0 < d < n$ be coprime. Then the  sheaf  of orders $\kA = \kA_{(n, d)}$ has the following  description:
\begin{equation}
\kA \cong  \CC \times \lieA/\sim, \quad \mbox{where} \quad (z, Z) \sim (z+1, \mathsf{Ad}_X(Z)) \sim
(z + \tau, \mathsf{Ad}_Y(Z)),
\end{equation}
$X$ and $Y$ are  matrices given by (\ref{E:matricesXY}) and 
$\mathsf{Ad}_T(Z) = T Z T^{-1}$ for $T \in \bigl\{X, Y\bigr\}$ and $Z \in \lieA$.
\end{proposition}

\smallskip
\noindent
For any $(k, l) \in I := \bigl\{1, \dots, n\bigr\} \times \bigl\{1, \dots, n\bigr\}$ denote
$
Z_{(k, l)} = Y^k X^{-l}$ and $
Z_{(k, l)}^\vee  = \frac{1}{n} X^{l} Y^{-k}.
$
Note that the operators $\mathsf{Ad}_X, \mathsf{Ad}_Y \in \End_{\CC}(\lieA)$ commute. Moreover,   $$
\mathsf{Ad}_X(Z_{(k, l)}) = \varepsilon^{k} Z_{(k, l)} \quad
\mbox{and} \quad
\mathsf{Ad}_{Y}(Z_{(k, l)}) = \varepsilon^{l} Z_{(k, l)}
$$
for any $(k, l) \in I$. As a consequence $\bigl(Z_{(k, l)}\bigr)_{(k, l) \in I}$ is a basis of 
$\lieA$.

\smallskip
\noindent
Let $\can: \lieA \otimes \lieA \lar \End_{\CC}(\lieA)$ be the canonical isomorphism
sending a simple tensor $Z' \otimes Z''$ to the linear map $Z \mapsto \tr(Z' \cdot Z) \cdot Z''$.
Then we have:
\begin{equation}\label{E:HeisBasis}
\can(Z_{(k,l)}^\vee \otimes Z_{(k, l)})(Z_{(k', l')}) = \left\{
\begin{array}{cl}
Z_{(k, l)} & \mbox{\rm{if}} \quad  (k', l') = (k, l) \\
0 & \mbox{\rm{otherwise}}.
\end{array}
\right.
\end{equation}

\noindent
Recall the expressions for the first and  third Jacobian theta-functions (see e.g. \cite{MumfordTheta}):
\begin{equation}\label{E:Theta}
\left\{\begin{array}{l}
\bar{\theta}(z) = \theta_1(z|\tau) = 2 q^{\frac{1}{4}} \sum\limits_{n=0}^\infty (-1)^n q^{n(n+1)}
\sin\bigl((2n+1)\pi z\bigr), \\
\theta(z) = \theta_3(z|\tau) = 1+  2 \sum\limits_{n=1}^\infty  q^{n^2}
\cos(2\pi nz),
\end{array}
\right.
\end{equation}
where $q = \exp(\pi i \tau)$. They are related by the following identity:
\begin{equation}\label{E:RelBetwTheta}
\theta\Bigl(z + \frac{1 + \tau}{2}\Bigr) = i \exp\Bigl(-\pi i \bigl(z + \frac{\tau}{4}\bigr)\Bigr) \bar{\theta}(z).
\end{equation}

\begin{lemma}\label{L:Theta} For any $x \in \CC$ consider the function $\varphi_x(w) = - \exp\bigl(- 2\pi i (w + \tau - x)\bigr)$. Then the  following results are true.
\begin{itemize}
\item The vector space
\begin{equation}
\left\{\CC \stackrel{f}\lar \CC
\left|
\begin{array}{l}
f \, \mbox{ \rm{is holomorphic}} \\
f(w+1) = f(w) \\
f(w+\tau) = \varphi_x(w) f(w)
\end{array}
\right.
  \right\}
\end{equation}
  is one-dimensional and generated by  $\theta_x(w) := \theta\bigl(w + \frac{1 + \tau}{2} - x\bigr)$.
  \item
We have:
$
\kE(\varphi_x) \cong \kO_E\bigl([x]\bigr).
$
\item For $a, b \in \mathbb{R}$ let $v = a\tau + b \in \CC$ and $[v] = \upsilon(v) \in E$. Then we have: 
\begin{equation}\label{E:FamilyDegreeZero}
\kE(\exp(-2\pi i v)) \cong 
\kO_E\bigl([0]-[v]\bigr).
\end{equation}
In these terms we also get a description of a universal family $\kL$ of degree zero line bundles on $E$.
\end{itemize}
\end{lemma}

\smallskip
\noindent
A proof of these statements can be for instance found in 
\cite{MumfordTheta} or \cite[Section 4.1]{BK4}.

\smallskip
\noindent
 Let $U \subset \CC$ be a small open neighborhood of $0$ and
$O = \Gamma(U, \kO_\CC)$ be the ring of holomorphic functions on $U$. Let $z$ be a coordinate on $U$,
$\CC \stackrel{\eta} \lar E$ be the canonical covering map, $\omega = dz \in H^0(E, \Omega)$ and 
$\Gamma(U, \kA) \stackrel{\xi}\lar \lieA \otimes_{\CC} O$ be the standard  trivialization induced by the automorphy
data $(\mathsf{Ad}_X, \mathsf{Ad}_Y)$. One can also define a trivialization 
$\zeta$  of the universal  family $\kL$ of degree zero line bundles on $E$ compatible with the isomorphisms
 (\ref{E:FamilyDegreeZero}).

\smallskip
\noindent
 Consider the following vector space
$$
\Sol\bigl((n, d), v, x\bigr) =
\left\{\CC \stackrel{F}\lar \lieA
\left|
\begin{array}{l}
F \mbox{\textrm{\quad is holomorphic}} \\
F(w+1) = \mathsf{Ad}_X\bigl(F(w)\bigr)\\
F(w+\tau) = \varphi_{x-v}(w) \mathsf{Ad}_Y\bigl(F(w)\bigr)
\end{array}
\right.
  \right\}.
  $$
\begin{proposition}\label{P:SolutionsElliptiques} The following diagram
\begin{equation}
\begin{array}{c}
\xymatrix{
\kA^v\big|_x \ar[d]_{\jmath^v_x} & & H^0\bigl(\kA^v(x)\bigr) 
\ar[ll]_-{\res^v_x} \ar[rr]^-{\ev^v_y} \ar[d]^{\jmath} & & \kA^v\big|_y \ar[d]^{\jmath^v_y} \\
\lieA & & \Sol\bigl((n, d), v, x\bigr) \ar[ll]_-{\overline{\res}_x} \ar[rr]^-{\overline{\ev}_y} &  & \lieA
}
\end{array}
\end{equation}
is commutative, where for $F \in \Sol\bigl((n, d), v,  x\bigr)$ we have:
\begin{equation*}
\overline{\res}_x(F) = \frac{F(x)}{\theta'\bigl(\frac{1 + \tau}{2}\bigr)}
\quad \mbox{\rm{and}} \quad
\overline{\ev}_y(F) = \frac{F(y)}{\theta\bigl(y - x + \frac{1 + \tau}{2}\bigr)}.
\end{equation*}
The isomorphisms of vector spaces  $\jmath^v_x, \jmath^v_y$ and $\jmath$ are  induced by the trivializations $\xi$ and $\zeta$ as well as the pull-back functor  $\eta^*$.
\end{proposition}

\noindent
\emph{Comment on the proof}. Since an  analogous result is proven in \cite[Corollary 4.2.1]{BK4}, we omit details here. \qed

\medskip
\noindent
Now we are prepared to prove the main result of this work.

\begin{theorem}\label{T:main} Let  $r_{((n, d), \tau)}(v; x, y)$ be the solution of AYBE corresponding to the datum $\bigl(E_\tau, \kA_{(n, d)}\bigr)$ with respect to the trivializations $\xi$ (respectively,  $\zeta$) of 
$\kA$ (respectively,  $\kL$) introduced above. Then it
 is given by the  expression (\ref{E:SolutionElliptic}).
\end{theorem}
\begin{proof}

We first compute an explicit basis of the vector space 
$\Sol\bigl((n, d),v,  x\bigr)$.
Let
$$
F(w) = \sum\limits_{(k, l) \in I} f_{(k, l)}(w)  Z_{(k, l)}.
$$
The condition $F \in \Sol\bigl((n, d), v, x\bigr)$ yields the following constraints on
 the coefficients $f_{(k, l)}$:
\begin{equation}\label{E:EllSystem}
\left\{
\begin{array}{lcl}
f_{(k, l)}(w+1) & = &  \varepsilon^k f_{(k, l)}(w) \\
f_{(k, l)}(w+\tau) & = &  \varepsilon^l \varphi_{x-v}(w) f_{(k, l)}(w).
\end{array}
\right.
\end{equation}
It follows from Lemma \ref{L:Theta} that the vector space of holomorphic solutions of the system (\ref{E:EllSystem})
is one-dimensional and  generated by the function
$$
f_{(k, l)}(w) = \exp\Bigl(- \frac{2 \pi i d}{n} kw\Bigr) \theta\Bigl(
w + \frac{1+ \tau}{2} + v - x - \frac{d}{n}(k \tau - l)\Bigr).
$$
From Proposition \ref{P:SolutionsElliptiques} and formula (\ref{E:HeisBasis}) it follows that $r_{((n, d), \tau)}(v; x, y)$ is given by the following expression:
\begin{equation*}
r_{((n, d), \tau)}(v; x, y) = \sum\limits_{(k, l) \in I}
r_{(k, l)}(v; z) \, 
Z_{(k, l)}^\vee \otimes Z_{(k, l)},
\end{equation*}
where $z = y - x$ and
$$
r_{(k, l)}(v; z) = \exp\Bigl(-\frac{2 \pi i d}{n} k z\Bigr)
\frac{\displaystyle \theta'\Bigl(\frac{1 + \tau}{2}\Bigr) \theta\Bigl(z + v + \frac{1 + \tau}{2}- \frac{d}{n}(k\tau - l)\Bigr)}{\displaystyle 
\theta\Bigl(v + \frac{1 + \tau}{2}- \frac{d}{n}(k\tau - l)\Bigr) \theta\Bigl(z + \frac{1 + \tau}{2}\Bigr)}.
$$
Relation (\ref{E:RelBetwTheta}) implies that
$$
\frac{\displaystyle \theta'\Bigl(\frac{1 + \tau}{2}\Bigr) \theta\Bigl(z + v + \frac{1 + \tau}{2}- \frac{d}{n}(k\tau - l)\Bigr)}{\displaystyle 
\theta\Bigl(v + \frac{1 + \tau}{2}- \frac{d}{n}(k\tau - l)\Bigr) \theta\Bigl(z + \frac{1 + \tau}{2}\Bigr)} =  
\frac{\displaystyle \theta'\Bigl(\frac{1 + \tau}{2}\Bigr) \bar{\theta}\Bigl(z+ v - \frac{d}{n}(k\tau - l)\Bigr)}{
\displaystyle i \exp\Bigl(-\pi i \frac{\tau}{4}\Bigr)\bar\theta\Bigl(v -\frac{d}{n}(k\tau - l)\Bigr) \bar\theta(z)}
$$
Moreover, it  follows from  (\ref{E:RelBetwTheta}) that $\theta'\Bigl(\dfrac{1 + \tau}{2}\Bigr) = 
i \exp\Bigl(-\pi i \dfrac{\tau}{4}\Bigr) \bar{\theta}'(0)$. Hence, we get:

\begin{multline*}
r_{(k, l)}(v; z) = \exp\Bigl(-\frac{2 \pi i d}{n} k z\Bigr)
\frac{\displaystyle \bar{\theta}'(0) \bar{\theta}\Bigl(z+ v - \frac{d}{n}(k\tau - l)\Bigr)}{
\displaystyle \bar\theta\Bigl(v -\frac{d}{n}(k\tau - l)\Bigr) \bar\theta(z)} \\
 = 
\exp\Bigl(-\frac{2 \pi i d}{n} k z\Bigr) \sigma\Bigl(v -\frac{d}{n}(k\tau - l), z\Bigr).
\end{multline*}
Here we use the fact  that 
the Kronecker elliptic function $\sigma(u, z)$ defined by  (\ref{E:KroneckerEllFunct}) satisfies the  formula: $
\sigma(u,  z)
 = \dfrac{\displaystyle \bar\theta'(0) \bar\theta_1(u+z)}{\displaystyle \bar\theta(u) \bar\theta(z)}
$
(see for instance \cite[Section 3]{Zagier}).
We have a bijection $\left\{1, \dots, n\right\} \lar \left\{0, \dots, n-1\right\}, k \mapsto (n-k)$.
Using this substitution as well as  the  identity
$ 
 \sigma(u - d\tau,  x) = \exp(2\pi i d z) \sigma(u, z),
$
we end up with the expression (\ref{E:SolutionElliptic}), as asserted.
\end{proof}

\begin{remark} Let $r(u; x_1, x_2)$ be a non-degenerate skew-symmetric solution of AYBE (\ref{E:AYBE}) satisfying (\ref{E:AnsatzAYBE}).
Let $\lieg = \mathfrak{sl}_n(\CC)$ and  $\lieA \stackrel{\pi}\lar \lieg, Z \mapsto Z - \frac{1}{n}\tr(Z) I_n$. Then
$$
\bar{r}(x_1, x_2) = (\pi \otimes \pi)(r_1(x_1, x_2))
$$
is a solution of the classical Yang--Baxter equation
\begin{equation*}
\left\{
\begin{array}{l}
\bigl[\bar{r}^{12}(x_1, x_2), \bar{r}^{13}(x_1, x_3)\bigr] + \bigl[\bar{r}^{13}(x_1, x_3), \bar{r}^{23}(x_2, x_3)\bigr] +
\bigl[\bar{r}^{12}(x_1, x_2),
\bar{r}^{23}(x_2, x_3)\bigr] = 0
\\
\bar{r}^{12}(x_1, x_2) = 
-\bar{r}^{21}(x_2, x_1),
\end{array} 
\right.
\end{equation*}
see \cite[Lemma 1.2]{Polishchuk1}.
Under certain additional assumptions (which are fulfilled provided $\bar{r}(x_1, x_2)$ is elliptic or trigonometric), the function $R(x_1, x_2) =  r(u_\circ; x_1, x_2)$
(where $u = u_\circ$ from the domain of definition of $r$ is fixed)  satisfies  the quantum Yang--Baxter equation 
\begin{equation*}\label{E:QYBE}
R(x_1, x_2)^{12} R(x_1, x_3)^{13} R(x_2, x_3)^{23} =
R(x_2, x_3)^{23} R(x_1, x_3)^{13} R(x_1, x_2)^{12},
\end{equation*}
see \cite[Theorem 1.5]{Polishchuk2}. In fact, the  expression (\ref{E:SolutionElliptic}) is a well-known elliptic solution of Belavin of the quantum Yang--Baxter equation; see \cite{Belavin}.
\end{remark}

\end{document}